\pdfoutput=1
\documentclass{shinyart}

\usepackage[utf8]{inputenc}

\usepackage{shinybib}

\usepackage{enumerate}
\usepackage{autonum}

\renewcommand{\phi}{\varphi}
\newcommand{\N}{\mathbb{N}}
\newcommand{\R}{\mathbb{R}}
\newcommand{\calF}{\mathcal{F}}
\newcommand{\calG}{\mathcal{G}}

\newcommand{\calA}{\mathcal{A}}

\newcommand{\calS}{\mathcal{S}}
\newcommand{\calH}{\mathcal{H}}

\newcommand{\norm}[1]{\| #1 \|}
\newcommand{\set}[2]{\left\{#1:#2\right\}}
\DeclareMathOperator{\sign}{\mathrm{sign}}

\DeclareMathOperator{\Id}{\mathrm{Id}}

\DeclareMathOperator*{\argmin}{\mathrm{arg\,min}}
\newcommand{\prox}{\mathrm{prox}}

\newcommand{\omobs}{\omega_\mathrm{obs}}
\newcommand{\lastgamma}{\bar{\gamma}}  	% smallest gamma in homotopy loop

\usepackage{booktabs}
\usepackage{graphicx}
\usepackage{tikz,pgfplots}
\usepgfplotslibrary{colorbrewer}
\pgfplotsset{compat=newest}
\pgfplotsset{plot coordinates/math parser=false}

\title{A convex penalty for switching control of partial differential equations}

\author{Christian Clason\thanks{Faculty of Mathematics, University Duisburg-Essen, 45117 Essen, Germany (\email{christian.clason@uni-due.de})}
    \and Armin Rund\thanks{Institute of Mathematics and Scientific Computing, University of Graz, Heinrichstrasse 36, 8010 Graz, Austria (\email{armin.rund@uni-graz.at}, \email{karl.kunisch@uni-graz.at}).}
    \and Karl Kunisch\footnotemark[2]
    \and Richard C.~Barnard\footnotemark[2]
}

\date{July 29, 2015}

\hypersetup{
    pdftitle={A convex penalty for switching control of partial differential equations},
    pdfauthor={Christian Clason, Armin Rund, Karl Kunisch, Richard C.~Barnard},
    pdfkeywords={Optimal control, Switching control, Partial differential equations, Nonsmooth optimization, Convex analysis, Semi-smooth Newton method}
}
\begin{document}

\maketitle
\allowdisplaybreaks

\begin{abstract}
    A convex penalty for promoting switching controls for partial differential equations is introduced; such controls consist of an arbitrary number of components of which at most one should be simultaneously active.
    Using a Moreau--Yosida approximation, a family of approximating problems is obtained that is
    amenable to solution by a semismooth Newton method. The efficiency of this approach and the structure of the obtained controls are demonstrated by numerical examples.
\end{abstract}

\section{Introduction}

Switching control refers to time-dependent optimal control problems with a vector-valued control of which at most one component should be active at every point in time.
We focus here on optimal tracking control for a linear diffusion equation
$Ly = Bu$ on $\Omega_T:=(0,T]\times \Omega$, $y(0)=y_0$ on $\Omega$,
where $L=\partial_t - A$ for an elliptic operator $A$ defined on $\Omega\subset \R^n$ carrying suitable boundary conditions. The control operator $B$ is defined by
\begin{equation}
    (Bu)(t,x)= \sum_{i=1}^N\chi_{\omega_i}(x) u_i(t),
\end{equation}
where $\chi_{\omega_i}$ is the characteristic function of the given control domain $\omega_i\subset \Omega$ of positive measure. Furthermore, let $\omobs\subset \Omega$ denote the observation domain and let $y^d\in L^2(0,T;L^2(\omobs))$ denote the target.
We consider the standard optimal control problem
\begin{equation}
    \left\{\begin{aligned}
            \min_{u \in L^2(0,T;\R^N)} &\frac12\norm{y-y^d}_{L^2(0,T;L^2(\omobs))}^2 + \frac\alpha2 \int_0^T|u(t)|_2^2\,dt,\\
            \text{s.\,t.}\quad & Ly = Bu, \quad y(0)=y_0,
    \end{aligned}\right.
\end{equation}
where $|v|_2^2 =\sum_{j=1}^N v_j^2$ denotes the squared $\ell^2$-norm on $\R^N$.
To promote the switching structure of the optimal control $\bar u\in L^2(0,T;\R^N)$, we suggest adding an additional penalty term
\begin{equation}
    \beta\int_0^T\sum_{\stackrel{i,j=1}{i< j}}^N |u_i(t) u_j(t)|\,dt
\end{equation}
with $\beta>0$ to the objective, which can be interpreted as an $L^1$-penalization of the switching constraint $u_i(t)u_j(t)=0$ for $i\neq j$ and $t\in[0,T]$. For the choice $\beta=\alpha$, the sum of the control cost and the penalty can be simplified to yield the problem
\begin{equation}\label{eq:problem_motiv}
    \left\{\begin{aligned}
            \min_{u \in L^2(0,T;\R^N)} &\frac12\norm{y-y^d}_{L^2(0,T;L^2(\omobs))}^2 + \frac\alpha2 \int_0^T|u(t)|_1^2\,dt,\\
            \text{s.\,t.}\quad & Ly = Bu, \quad y(0)=y_0,
    \end{aligned}\right. \tag{P}
\end{equation}
where $|v|_1 = \sum_{j=1}^N |v_j|$ denotes the $\ell^1$-norm on $\R^N$.
This is a convex optimization problem, for which we derive first-order optimality conditions in primal-dual form whose Moreau--Yosida regularization (to be introduced below) can be solved using a superlinearly convergent semismooth Newton method. The effect of other choices for $\beta$ will be discussed in \cref{sec:multiswitching:relation}.

The approach which we follow here is related to both the switching control problem in \cite{CIK:2014} and the distributed parabolic sparse control problem in \cite{CCK:2012}.
In \cite{CIK:2014} a nonconvex formulation in the case where $N=2$ was considered; we compare its convex relaxation to the present approach below. 
One advantage of the approach presented in this paper over that in \cite{CIK:2014} is given by the fact that there is no significant additional technical burden when considering switching between $N>2$ controls.
In \cite{CCK:2012} the $L^2$ norm in time of the measure norm in space was used to promote temporally varying sparsity in space. While the choice of the nonsmooth functional involving the controls in \eqref{eq:problem_motiv} is motivated by sparsity considerations, one can arrive at this functional also from controllability--observability considerations. In fact, it was shown in
\cite[Theorem 4.1]{Zuazua:2011} that -- provided an appropriately defined controllability Gramian has full rank -- exact null controls with perfect switching have minimal $L^2(0,T;\R^2)$ norm, where $\R^2$ is endowed with the $\ell^1$ norm.
In contrast, \cite{Hante:2013} follows a different approach where binary or integer decision variables are sought within a relaxation technique combined with a suitable rounding strategy. 

Let us comment on further related work. While our work here aims at formulating 
optimal controls problems with switching controls in a way that allows an efficient numerical treatment, the larger body of work focuses on the stabilization of switching systems. For ordinary differential equations we refer to  e.g., \cite{Dolcetta:1984,Liberzon:2003,Shorten:2007}.
For partial differential equations, this problem has received comparatively little attention.
In both cases, one should distinguish switching control in the sense defined above from the control of switched systems.
For the latter in the context of PDEs, we refer to \cite{Hante:2009,Stojanovic:1989}. 
In \cite{Hante:2011}, converse Lyapunov theorems for abstract switched systems are developed. Lyapunov techniques are also used to study switches in hyperbolic systems in \cite{Prieur:2014}, and existence results for optimal control of switching systems modelling the use of bacteria for pollution removal are obtained in \cite{Seidman}.
Exact null controls with switching structure for the heat and wave equation were treated in \cite{Zuazua:2011,Lu:2013,Wang:2014} and \cite{Gugat:2008}, respectively.

This work is organized as follows. \Cref{sec:convex} discusses the existence and first-order optimality conditions
for solutions to \eqref{eq:problem_motiv} as well as its regularization within an abstract convex analysis framework.
Explicit pointwise characterizations of the switching relations arising from the optimality system and for its regularization are given in \cref{sec:multiswitching}.
Here we also discuss the relation of the proposed switching functional in \eqref{eq:problem_motiv} to other possible choices of the penalty term.
\Cref{sec:newton} is concerned with the numerical solution of the regularized optimality system by a semismooth Newton method.
A numerical example for switching control of a two-dimensional linear heat equation is computed in \cref{sec:results}.

\section{Convex analysis approach}\label{sec:convex}

We recall the convex analysis approach for (nonconvex) switching controls for partial differential equations from \cite{CIK:2014}, which is also applicable to the convex penalty considered here.
For this purpose, we consider Problem \eqref{eq:problem_motiv} in the reduced form
\begin{equation}
    \min_u \calF(u)+\calG(u),
\end{equation}
with $\calF:L^2(0,T;\R^N)\to \R$ and $\calG:L^2(0,T;\R^N)\to \R$
given by
\begin{equation}
    \calF(u)= \frac{1}{2} \|Su-y^d\|_{L^2(0,T;L^2(\omobs))}^2, \qquad
    \calG(u) = \frac\alpha 2\int_0^T |u(t)|_1^2\,dt,
\end{equation}
where the continuous affine solution operator $S: u\mapsto y$ assigns to any control $u \in L^2(0,T;\R^N)$
the unique state $y \in L^2(\Omega_T)$ satisfying the state equation $Ly=Bu$ with initial condition $y(0)=y_0$ subject to appropriate boundary conditions.
Here we assume that the coefficients of $A$, the boundary and initial conditions as well as the domain $\Omega$ are sufficiently regular that the range of $S$ is contained in
\begin{equation}
    \label{eq:range_S}
    W(0,T):=L^2(0,T;H^1_0(\Omega))\cap W^{1,2}(0,T;H^{-1}(\Omega))\hookrightarrow C([0,T];L^2(\Omega)).
\end{equation}

Since $S$ is affine, $\calF$ is proper, convex and lower semicontinuous. Furthermore, since the squared norm $|\cdot|_1^2$ is convex,
$\calG$ is proper, convex, lower semicontinuous, and, in addition, radially unbounded. Existence of a solution thus follows from standard arguments, e.g., Tonelli's direct method.
\begin{proposition}
    There exists a minimizer $\bar{u}$ for Problem \eqref{eq:problem_motiv}.
\end{proposition}
We next derive first-order optimality conditions in primal-dual form.
Throughout, for any proper convex function $\calH$, we denote by $\calH^*$ its
Fenchel conjugate and by $\partial \calH$ its subdifferential; see, e.g., \cite{Bauschke:2011, Ekeland:1999a} for their definitions.
The following proposition is a direct consequence of the sum rule and inversion formula for convex subdifferentials (see, e.g., \cite[Corollary 16.24]{Bauschke:2011} for the latter) as well as the Fréchet-differentiability of $\calF$.
\begin{proposition}
    The control $\bar{u}\in L^2(0,T;\R^N)$ is a minimizer for \eqref{eq:problem_motiv}
    if and only if there exists a $\bar{p}\in L^2(0,T;\R^N)$ such that
    \begin{equation}\label{eq:formal_opt}
        \left\{\begin{aligned}
                -\bar p &= \calF'(\bar u),\\
                \bar u &\in \partial\calG^*(\bar p),
        \end{aligned}\right.
        \tag{OS}
    \end{equation}
    holds.
\end{proposition}
Since $\calF$ is a standard quadratic tracking term, the first relation in \eqref{eq:formal_opt} can be expressed in a straightforward manner in terms of the solution operator $S=L^{-1}B$
and its adjoint $S^*=B^*L^{-*}$ (with homogeneous boundary and initial conditions), i.e., $\bar p=-S^*(S\bar u-y^d)$.
For later use, we point out that due to \eqref{eq:range_S} and the specific choice of $S$ there holds $\bar p\in V:=B^*(W(0,T)) \hookrightarrow L^{r}(0,T;\R^N)$ for any $r>2$.

The second relation is responsible for the switching structure of the optimal control $\bar u$, and we will give a pointwise characterization in \cref{thm:subdiff} below.

\bigskip

Our numerical approach is based on the Moreau--Yosida regularization of \eqref{eq:formal_opt}. Specifically, we replace $\partial\calG^*$ for $\gamma>0$ by
\begin{equation}
    \partial\calG^*_\gamma(p) := (\partial\calG^*)_\gamma(p) := \frac{1}{\gamma} \left(p-\prox_{\gamma \calG^*}(p)\right),
\end{equation}
where
\begin{equation}
    \prox_{\gamma \calG^*}(v) := \argmin_{w\in L^2(0,T;\R^N)}\frac1{2\gamma}\norm{w-v}_{L^2(0,T;\R^N)}^2 + \calG^*(w) = \left(\Id + \gamma\partial\calG^*\right)^{-1}(v)
\end{equation}
is the proximal mapping of $\calG^*$, which in Hilbert spaces coincides with the resolvent of $\partial\calG^*$. Note that the proximal mapping and
thus the Moreau--Yosida regularization of a proper and convex functional is always single-valued and Lipschitz continuous; see, e.g., \cite[Corollary~23.10]{Bauschke:2011}.

We then consider the regularized system
\begin{equation}\label{eq:formal_opt_reg}
    \left\{\begin{aligned}
            - p_\gamma &= \calF'(u_\gamma),\\
            u_\gamma &= \partial\calG^*_{\gamma}(p_\gamma).
    \end{aligned}\right.
\tag{OS$_\gamma$}
\end{equation}
Again, we will give an explicit formulation of \eqref{eq:formal_opt_reg} in the next section.
\begin{proposition}
    For each $\gamma>0$ system \eqref{eq:formal_opt_reg} admits a unique solution $(u_\gamma, p_\gamma)$.
\end{proposition}
\begin{proof}
    Using convex analysis techniques (see, e.g., \cite[Chapter 12]{Bauschke:2011}), we obtain that
    \eqref{eq:formal_opt_reg} is the necessary optimality condition for
    \begin{equation}\label{objective_reg_con}
        \min_{u} \calF(u)+\left(\calG^*_\gamma\right)^*(u).
    \tag{P$_\gamma$}
\end{equation}
In fact, since $\calF$ is globally defined and continuous, we have the necessary optimality condition
\begin{equation}
    0 \in \partial \calF(u_\gamma)+\partial\left(\calG^*_\gamma\right)^*(u_\gamma).
\end{equation}
Setting $p_\gamma := -\calF'(u_\gamma)$ we recall again the subdifferential inversion formula from \cite[Corollary 16.24]{Bauschke:2011}.
Additionally, we note that the Yosida-regularization $\partial\calG^*_\gamma$ of the subdifferential $\partial\calG^*$
coincides with the Fréchet derivative of the Moreau-envelope $\calG^*_\gamma$ of $\calG^*$; see, e.g., \cite[Proposition 12.29]{Bauschke:2011}.
Together, these yield \eqref{eq:formal_opt_reg}.

From \cite[Remark 12.24]{Bauschke:2011}, we have the alternative characterization of the Moreau-envelope via the infimal convolution
\begin{equation}
    \calG^*_\gamma = \calG^* \mathop{\Box} \frac{1}{2\gamma}\norm{\cdot}_{L^2(0,T;\R^N)}^2,
\end{equation}
where $\R^N$ is endowed with the Euclidean norm.
Furthermore, from \cite[Theorem 15.3]{Bauschke:2011}, we obtain
\begin{equation}
    \begin{aligned}
        (\calG^*_\gamma)^* &= \left(\calG^{*} \mathop{\Box} \frac1{2\gamma} \norm{\cdot}_{L^2(0,T;\R^N)}^2\right)^*
        = \left(\left(\calG + \frac\gamma2\norm{\cdot}_{L^2(0,T;\R^N)}^2\right)^*\right)^*\\
        &= \calG + \frac\gamma2 \norm{\cdot}_{L^2(0,T;\R^N)}^2.
    \end{aligned}
\end{equation}
This implies that $(\calG^*_\gamma)^*$ is strictly convex.
Moreover, $\calG^*_\gamma \leq \calG^*$ and hence $0\leq \calG^{**}=\calG\leq (\calG^*_\gamma)^*$.
Therefore a minimizing sequence for \eqref{objective_reg_con} is necessarily
bounded, and by a weak subsequential limit argument, existence of a unique solution $u_\gamma$ to \eqref{objective_reg_con} follows.
Together with $p_\gamma =  -\calF'(u_\gamma)$, we thus obtain a unique solution $(u_\gamma,p_\gamma)$ to \eqref{eq:formal_opt_reg}.
\end{proof}

Finally, convergence as $\gamma\to 0$ can be shown as in \cite[Proposition 2.5]{CIK:2014}. This requires showing that
$\{\partial\calG^*(p_\gamma)\}_{\gamma>0}$ is uniformly bounded in $\gamma$ provided that $\{p_\gamma\}_{\gamma>0}$ is bounded.
But this follows from the explicit characterization in \eqref{eq:subdiff_N} below.
\begin{proposition}
    For any sequence $\gamma_n\to0$, the sequence $\{(u_{\gamma_n},p_{\gamma_n})\}_{n\in\N}$ converges weakly to a solution $(\bar u, \bar p)$ to \eqref{eq:formal_opt}.
\end{proposition}

\section{Switching penalty}\label{sec:multiswitching}

Here we  compute an explicit, pointwise, characterization of $\partial\calG^*$ and $\partial\calG^*_\gamma$ by exploiting the fact that for integral functionals of the form
\begin{equation}
    \calG(u) = \int_0^T g(u(t))\,dt,
\end{equation}
the Fenchel conjugate and convex subdifferential can be computed pointwise via  $g$; see, e.g., \cite[Props.~IV.1.2, IX.2.1]{Ekeland:1999a}, \cite[Prop.~16.50]{Bauschke:2011}. We thus focus on
\begin{equation}
    g:\R^N\to \R,\qquad g(v) = \frac\alpha2|v|_1^2.
\end{equation}
Because the Fenchel conjugate of half the squared norm is half the squared dual norm (see, e.g., \cite[Example 3.27]{Boyd:2004}), a scaling argument (e.g., \cite[Proposition 13.20\,(i)]{Bauschke:2011}) yields
\begin{equation}\label{eq:conj}
    g^*:\R^N\to\R,\qquad g^*(q) = \frac1{2\alpha}|q|_\infty^2 = \max_{1\leq i\leq N} \frac1{2\alpha} q_i^2.
\end{equation}

\subsection{Subdifferential}
Since $g^*$ is the maximum of finitely many convex and differentiable functions, its convex subdifferential is given by
\begin{equation}\label{eq:conj_subdiff}
    \partial g^*(q) = {\mathrm{co}} \left(\bigcup_{\{i:g^*(q)=g_i^*(q)\}}\left\{(g_i^*)'(q) \right\}\right),
\end{equation}
where ${\mathrm{co}}$ denotes the closed convex hull and $g_i^*(q) = \frac1{2\alpha}q_i^2$; see, e.g., \cite[Corollary 4.3.2]{Hiriart:2001}.
It is instructive to first consider the case $N=2$, for which only the three cases $g^*(q) = g_1(q)\neq g_2(q)$, $g^*(q) = g_2(q)\neq g_1(q)$ and $g^*(q) = g_1(q)=g_2(q)$ need to be considered. The corresponding convex hulls are then given by
\begin{equation}
    \label{eq:subdiff_2}
    \partial g^*(q) = \begin{cases}
        \left\{\left(\frac1\alpha q_1,0\right)\right\}& \text{if }|q_1|>|q_2|,\\
        \left\{\left(0,\frac1\alpha q_2\right)\right\}& \text{if }|q_1|<|q_2|,\\
        \left\{t\left(\frac1\alpha q_1,0\right)+(1-t)\left(0,\frac1\alpha q_2\right):t\in[0,1]\right\}& \text{if }|q_1|=|q_2|.
    \end{cases}
\end{equation}
For the general case, we proceed similarly by considering all possible cases. It is then straightforward to verify that the subdifferential can be given componentwise for $1\leq j\leq N$ as
\begin{equation}
    \label{eq:subdiff_N}
    [\partial g^*(q)]_j = \begin{cases}
        \left\{\frac1\alpha q_j\right\} & \text{if }|q_j| = \max_i |q_i|\text{ and } |q_j|> |q_i| \text{ for } i\neq j,\\
        \left\{0\right\} & \text{if }|q_j| < \max_i |q_i|,\\
        \left\{\frac{s_j}\alpha q_j:s_j\geq 0,\sum_{i\in\calA} s_i = 1\right\} & \text{if }|q_j| = \max_i |q_i|\text{ and } \exists i\neq j: |q_j| = |q_i|,
    \end{cases}
\end{equation}
where $\calA:=\set{j}{|q_j|=\max_i|q_i|}$. We thus obtain for $u,p\in L^2(0,T;\R^N)$ the pointwise characterization
\begin{equation}
    \partial\calG^*(p) = \set{u\in L^2(0,T;\R^N)}{u(t) \in \partial g^*(p(t)) \text{ for a.e. } t\in(0,T)}.
\end{equation}
In particular, this yields a pointwise characterization of the second relation in \eqref{eq:formal_opt}.
\begin{proposition}\label{thm:subdiff}
    The minimizer $\bar u\in L^2(0,T;\R^N)$  of \eqref{eq:problem_motiv} and $\bar p := -\calF'(\bar u)\in L^2(0,T;\R^N)$ satisfy for almost every $t\in (0,T)$ and $1\leq j\leq N$
    \begin{equation}
        \label{eq:switching_rel}
        \bar u_j(t) \in \begin{cases}
            \left\{\frac1\alpha \bar p_j(t)\right\} & \text{if }|\bar p_j(t)| = \max_i |\bar p_i(t)|,\ |\bar p_j(t)|> |\bar p_i(t)| \text{ for } i\neq j,\\
            \left\{0\right\} & \text{if }|\bar p_j(t)| < \max_i |\bar p_i(t)|,\\
            \left\{\frac{s_j}\alpha \bar p_j(t):s_j\geq 0,\sum_{i\in\calA(t)} s_i = 1\right\} & \text{if }|\bar p_j(t)| = \max_i |\bar p_i(t)|,\ \exists i\neq j: |\bar p_j(t)| = |\bar p_i(t)|,
        \end{cases}
    \end{equation}
    where $\calA(t):=\set{j}{|\bar p_j(t)| = \max_i |\bar p_i(t)|}$.
\end{proposition}
From this proposition, the desired switching property follows: Outside of the singular arc
\begin{equation}
    \calS:=\set{t\in(0,T)}{|\calA(t)|>1},
\end{equation}
we have $\bar u_{j_1}\bar u_{j_2} = 0$ for all $j_1\neq j_2$.
In particular, if $\calS$ has Lebesgue measure zero, $\bar u$ exhibits perfect switching, i.e., $\bar u_{j_1}\bar u_{j_2} = 0$ almost everywhere.
Under the stronger assumption that $|\bar p_{j_1}(t)| \neq |\bar p_{j_2}(t)|$ for all $j_1\neq j_2$ and almost all $t\in (0,T)$, a similar relation for exact null controls was given in \cite[equations (2.36--2.38)]{Zuazua:2011}.

\begin{remark}
    It is difficult to give general conditions for the optimal control to be perfectly switching; instead one would have to exploit properties of the specific (adjoint) state equation.
    For example, in the case of the one-dimensional heat equation, one could use the fact that solutions are real-analytic with respect to time to argue that $|\bar p_i(t)| = |\bar p_j(t)|$ for all $t\in O$ for an open set $O\subset(0,T)$ is only possible if $\bar p_i(t) = \bar p_j(t) = 0$ for all $t\in O$; cf.~\cite[Section 6.1]{Zuazua:2011}. 
\end{remark}

\subsection{Proximal mapping}

We next characterize the proximal mapping $\prox_{\gamma g^*}(p)$ via the resolvent $(\Id+\gamma\partial g^*)^{-1}(p)$ by proceeding similarly to \cite[Section 3.3]{CIK:2014}; see also \cite{Kunisch:2008a}.
\begin{proposition}\label{prop_proximal_mapping}
    Let the components of $v\in \R^N$ be sorted by decreasing magnitude and let $d$ be the smallest index for which
    \begin{equation}\label{eq:prox_d}
        |v_{d+1}| <  \frac\alpha{d\alpha+\gamma}\sum_{i=1}^{d} |v_i|
    \end{equation}
    holds; if no such index exists, let $d=N$. The explicit form of the proximal mapping is then given componentwise for $1\leq j\leq N$ by
    \begin{align}
        [\prox_{\gamma g^*}(v)]_j = \begin{cases}
            \sign(v_j) \tfrac\alpha{d\alpha+\gamma}\sum_{i=1}^d |v_i| & \text{if }j\leq d,\\
            v_j & \text{if }j>d.
        \end{cases}
    \end{align}
\end{proposition}
\begin{proof}
    Let $w:=(\mathrm{Id}+\gamma\partial g^*)^{-1}(v)$.
    This is characterized by the subdifferential inclusion
    \begin{equation}\label{eq:resolvent}
        v\in (\mathrm{Id}+\gamma\partial g^*)(w) = \{w\}+\gamma\partial g^*(w).
    \end{equation}
    First, we observe that \eqref{eq:resolvent} implies that $\sign(v_i) = \sign(w_i)$.
    We next follow the case discrimination in the characterization of the subdifferential, where we
    assume that the components of $w$ are ordered by magnitude, i.e., $|w_1| \geq  |w_2| \geq \dots \geq |w_N|$
    (the remaining cases following by permutation).
    \begin{enumerate}[(i)]
        \item $|w_1|>|w_2|$:
            In this case, all components of the subdifferentials are single-valued; solving for each $w_i$ in \eqref{eq:resolvent} yields
            \begin{equation}
                w_i = \begin{cases} \frac\alpha{\alpha+\gamma} v_i & \text{if }i=1,\\
                    v_i & \text{if }i>1.
                \end{cases}
            \end{equation}
            The assumption $|w_1|>|w_2|$ is then equivalent to the condition
            \begin{equation}
                |v_1|>\left(1+\tfrac\gamma\alpha\right)|v_2|.
            \end{equation}
        \item $|w_1|=|w_2|>|w_3|$: Using that $\sign(w_i)=\sign(v_i)$, this implies that $w_1 = \sign(v_1)\bar w$ and $w_2 = \sign(v_2)\bar w$ for some $\bar w>0$.
            Inserting this into the set-valued case of the subdifferential, we deduce that for some $t\in[0,1]$,
            \begin{align}
                \sign(v_1)|v_1| &= \left(1+t\tfrac\gamma\alpha\right)\sign(v_1)\bar w,\\
                \sign(v_2)|v_2| &= \left(1+(1-t)\tfrac\gamma\alpha\right)\sign(v_2)\bar w.
            \end{align}
            Dividing by the $\sign(v_i)$ on both sides and adding yields
            \begin{equation}
                |v_1| + |v_2| = \left(2+\tfrac\gamma\alpha\right)\bar w,
            \end{equation}
            from which we obtain
            \begin{equation}
                w_i = \begin{cases}
                    \sign(v_i)\tfrac\alpha{2\alpha+\gamma}(|v_1|+|v_2|), &\text{if } i \leq 2,\\
                    v_i &\text{if } i>2,
                \end{cases}
            \end{equation}
            the single-valued cases for $i>2$ following as before.

            Inserting these values into $|w_2|>|w_3|$ then yields the condition
            \begin{equation}
                |v_1|+|v_2| > \left(2+\tfrac\gamma\alpha\right)|v_3|.
            \end{equation}
            Note that in addition,
            \begin{equation}
                |v_1| \leq \left(1+\tfrac\gamma\alpha\right)|v_2|
            \end{equation}
            must hold since due to case (i), the converse is equivalent to $|w_1|>|w_2|$, in contradiction to the assumption.
        \item  $|w_1|=|w_2|=|w_3|>|w_4|$:
            As before, the chain of equalities is equivalent to
            \begin{align}
                |v_1| &= \left(1+t_1\tfrac\gamma\alpha\right)\bar w,\\
                |v_2| &= \left(1+t_2\tfrac\gamma\alpha\right)\bar w,\\
                |v_3| &= \left(1+(1-t_1-t_2)\tfrac\gamma\alpha\right)\bar w
            \end{align}
            for some $t_1,t_2$ with $t_1+t_2\in(0,1)$. Adding these equations yields that
            \begin{equation}
                w_i = \begin{cases}
                    \sign(v_i)\tfrac\alpha{3\alpha+\gamma}(|v_1|+|v_2|+|v_3|), &\text{if } i \leq 3,\\
                    v_i & \text{if }i>3,
                \end{cases}
            \end{equation}
            as well as the condition
            \begin{equation}
                |v_1|+|v_2|+|v_3| > \left(3+\tfrac\gamma\alpha\right)|v_4|.
            \end{equation}
            Again,
            \begin{equation}
                |v_1| \leq \left(1+\tfrac\gamma\alpha\right)|v_2|,
                \qquad\text{and}\qquad
                |v_1|+|v_2| \leq \left(2+\tfrac\gamma\alpha\right)|v_3|
            \end{equation}
            must hold since otherwise we obtain a contradiction in the case distinction.
        \item The remaining cases $|w_1|=\dots=|w_j|>|w_{j+1}|$ for $j=4,\dots,N$ follow analogously,
            establishing the desired characterization of the proximal mapping.
            \qedhere
    \end{enumerate}
\end{proof}

It is now straight-forward to give a pointwise characterization of the Moreau--Yosida regularization $\partial\calG^*_\gamma$ via
\begin{equation}
    [\partial\calG^*_\gamma(p)](t) = \partial g^*_\gamma (p(t)) = \frac1\gamma \left(p(t)- \prox_{\gamma g^*}(p(t))\right) \quad \text{for a.e. }t\in (0,T).
\end{equation}
We point out that for every $t\in (0,T)$, the number $d$ in \cref{prop_proximal_mapping} is the number of nonzero components in $u_\gamma(t)=[\partial\calG^*_\gamma(p)](t)$.
In particular for the case $d=1$, the explicit form of $\partial g^*_\gamma$ is given componentwise as
\begin{equation}
    \label{eq:h_gamma_d1}
    [\partial g^*_\gamma(q)]_j =
    \begin{cases}
        \frac1{\alpha+\gamma} q_j &\text{if } |q_j| = \max_i |q_i|\text{ and } |q_j|>\left(1+\tfrac\gamma\alpha\right) |q_i|, i\neq j,\\
        0 &\text{if } |q_j| < \max_i |q_i|.
    \end{cases}
\end{equation}
(Note that the missing case $|q_j|\leq \left(1+\tfrac\gamma\alpha\right) |q_i|$ is excluded by the assumption that $d=1$.)
Thus if for some $\gamma>0$, the solution $(u_\gamma,p_\gamma)$ to \eqref{eq:formal_opt_reg} is such that $d=1$ for almost every $t\in(0,T)$, the regularized control is perfectly switching.

\subsection{Relation to other approaches for switching control}\label{sec:multiswitching:relation}

In this section, we address the relation of the proposed approach to alternative $\ell^1$-type penalizations as well as to the $\ell^0$ formulation of \cite{CIK:2014}.
For this purpose, we recall the original introduction of $g$ as a combination of a quadratic control cost and an $\ell^1$-penalization of the switching constraint, i.e., as
\begin{equation}
    g_\beta(v) := \frac{\alpha}2 |v|_2^2 + {\beta} \sum_{\stackrel{i,j=1}{i< j}}^N |v_iv_j|.
\end{equation}
We now compare this formulation with other approaches, setting $N=2$ for the sake of simplicity.

First, we address the choice $\beta\neq \alpha$.
For $\beta<\alpha$, we can introduce $\gamma:=\alpha-\beta>0$ and rewrite $g_\beta$ as
\begin{equation}
    g_\beta(v) = \frac{\beta}2 |v|_2^2 + \beta |v_1v_2| +  \frac{\gamma}2 |v|_2^2 = (g^*_\gamma)^*(v)
\end{equation}
(with $\beta$ in place of $\alpha$ in the definition of $g$). The optimality system in this case is therefore equivalent to \eqref{eq:formal_opt_reg}.
In case $\beta > \alpha$, we can reformulate again to
\begin{equation}
    g_\beta(v) = \frac{\alpha-\beta}2 |v|_2^2 + \frac\beta2 |v|_1^2,
\end{equation}
which is obviously nonconvex. Computing the Fenchel conjugate leads to
\begin{equation}
    g_\beta^*(q) =
    \begin{cases}
        \frac1{2\alpha} q_1^2 & \text{if } |q_1|\geq |q_2|,\\
        \frac1{2\alpha} q_2^2 & \text{if } |q_1|\leq |q_2|,
    \end{cases}
\end{equation}
which coincides with \eqref{eq:conj}. 
We thus have that $g_\beta^*=g^*$, from which it follows that
\begin{equation}\label{eq:switching_relaxation}
    g = g^{**} = (g^*)^* = (g_\beta^*)^* = g_\beta^{**},
\end{equation}
i.e., $g$ is the lower convex envelope of the nonconvex function $g_\beta$.

\bigskip

We can also compare to the nonconvex $\ell^0$-switching formulation from \cite{CIK:2014},
\begin{equation}
    \label{eq:switching_l0}
    g_0(v) = \frac{\alpha}2 |v|_2^2 + \beta |v_1v_2|_0
\end{equation}
(where $|s|_0 = 0$ if $s=0$ and $1$ otherwise), whose Fenchel conjugate is given by
\begin{equation}
    \label{eq:switching_l0_subdiff}
    g_0^*(q) =
    \begin{cases}
        \frac1{2\alpha} q_1^2 & \text{if } |q_1|\geq |q_2| \text{ and }|q_2|\leq \sqrt{2\alpha\beta},\\
        \frac1{2\alpha} q_2^2 & \text{if } |q_1|\leq |q_2| \text{ and }|q_1|\leq \sqrt{2\alpha\beta},\\
        \frac1{2\alpha} (q_1^2+q_2^2)-\beta & \text{if } |q_1|,|q_2| \geq \sqrt{2\alpha\beta}.
    \end{cases}
\end{equation}
For fixed $q$ and $\beta$ sufficiently large, this coincides with \eqref{eq:subdiff_2}, since the third case
(corresponding to the free arc, where both control components are active) is never attained.
Therefore, the functional $g$ can be interpreted
as the convex relaxation of $g_0$ in the limit as $\beta\to\infty$, in the following sense: Taking the formal limit $\beta\to\infty$ in \eqref{eq:switching_l0},
we arrive at the (still nonconvex) constrained functional
\begin{equation}
    g_\infty(v) := \frac\alpha2 |v|_2^2 + \delta_{\{v:v_1v_2 = 0\}}(v) :=
    \begin{cases}
        \frac\alpha2 |v|_2^2 & \text{if }v_1v_2 = 0,\\
        \infty &\text{else,}
    \end{cases}
\end{equation}
where the switching property $v_iv_j =0$ for all $i\neq j$ is included as an explicit constraint.
Proceeding as in \cite[\S{}\,3.2]{CIK:2014}, we find that the Fenchel conjugate of $g_\infty$ is given by the first two cases of \eqref{eq:switching_l0_subdiff}.
As in the case of $g_\beta$, this implies that $g$ is the lower convex envelope of $g_\infty$.

\bigskip

We conclude that our choice of $g$ is a natural convex formulation for promoting switching controls. Convexification naturally entails that switching cannot be guaranteed for arbitrarily small choices of $\alpha$ but is crucial for an efficient numerical solution as discussed in the next section.

\section{Numerical solution}\label{sec:newton}

We apply a semismooth Newton method (see, e.g., \cite{Kunisch:2008a,Ulbrich:2011}) to \eqref{eq:formal_opt_reg}.
By eliminating $u_\gamma$ via the second relation of \eqref{eq:formal_opt_reg} and introducing the state $y_\gamma:=S(u_\gamma)\in Y:=L^2(\Omega_T)$, we arrive at the equivalent system
\begin{equation}\label{eq:optimality_red}
    \left\{\begin{aligned}
            y_\gamma &= S H_\gamma(p_\gamma),\\
            p_\gamma &= -S^* (y_\gamma-y^d),
    \end{aligned}\right.
\end{equation}
where we set $H_\gamma := \partial\calG_\gamma^*$.
We similarly set $h_\gamma := \partial g_\gamma^*$ such that
\begin{equation}
    [H_\gamma(p)](t) = h_\gamma(p(t)) \qquad\text{for a.e. } t\in (0,T).
\end{equation}
Exactly as in \cite[\S\,5]{CIK:2014}, one argues that $h_\gamma:\R^N\to\R^N$ is semismooth, and a Newton derivative $D_N h_\gamma(q)\in\R^{N\times N}$ is obtained by choosing
any element of the Clarke derivative $\partial_C h_\gamma(q)$ (which for the piecewise differentiable function $h_\gamma$ is given by the convex hull of the piecewise derivatives).
Here we take, assuming again for simplicity that the components of $q$ are sorted descending by magnitude, componentwise for $1\leq i,j\leq N$
\begin{equation}
    [D_N h_\gamma(q)]_{ji} =
    \begin{cases}
        \frac{(d-1)\alpha+\gamma}{\gamma(d\alpha+\gamma)} & \text{if } j=i\leq d,\\
        -\frac{\alpha}{\gamma(d\alpha+\gamma)}\sign(q_jq_i) & \text{if } j\leq d,\ i\leq d,\ i\neq j,\\
        0 &\text{if } j>d \text{ or } i>d,
    \end{cases}
\end{equation}
where $d$ is as in \cref{prop_proximal_mapping}. The Newton derivative for arbitrary $q\in\R^N$ can be obtained from this by appropriate permutation of rows and columns. Note that $D_N h_\gamma(q)$ is symmetric.
Furthermore, we point out that evaluation of the proximal mapping amounts to sorting for each $t\in(0,T)$ the vector $p(t)\in\R^N$,
which can be carried out in $\mathcal{O}(N\log N)$ operations,
and testing the $N-1$ conditions in \eqref{eq:prox_d} for the sorted vector, which requires $\mathcal{O}(N)$ operations.
Hence, evaluating $h_\gamma$ and assembling the Newton derivative $D_N h_\gamma$ can be performed in $\mathcal{O}(N\log N)$ operations,
thus avoiding the exponential complexity involved in computing switching points between any two (or more) control components; see, e.g., \cite{Iftime:2009}. A different relaxation approach for mixed-integer optimal control that also avoids exponential complexity was presented in \cite{Hante:2013}.

Since $h_\gamma$ is semismooth from $\R^N$ to $\R^N$, the corresponding superposition operator $H_\gamma$ is semismooth from $V\hookrightarrow L^r(0,T;\R^N)$ for $r>2$ to $L^2(0,T;\R^N)$, with Newton derivative at $p$ in direction $\delta p$ given by
\begin{equation}
    [D_N H_\gamma (p) \delta p](t) = D_N h_\gamma(p(t))\delta p(t) \qquad\text{for a.e. }t\in (0,1).
\end{equation}
Considering the system~\eqref{eq:optimality_red} as an operator equation from $Y\times V$ to $Y\times V$, a semismooth Newton step for its solution
thus consists of computing $(\delta y,\delta p)\in Y\times V$ for given $(y^k,p^k)\in Y\times V$ such that
\begin{equation}\label{eq:newton_step}
    \left\{\begin{aligned}
            \delta y - S_0 D_NH_\gamma(p^k)\delta p &= -y^k+ S H_\gamma(p^k) ,\\
            \delta p + S^*\delta y &= -p^k - S^* (y^k-y^d),
    \end{aligned}\right.
\end{equation}
and setting $y^{k+1} = y^k +\delta y$ and $p^{k+1} = p^k + \delta p$. Here, $S_{0}$ denotes the solution operator of the state equation with homogeneous initial and boundary conditions.
As in \cite[Proposition 5.2]{CIK:2014}, one shows uniformly bounded invertibility of each Newton step, from which  locally superlinear convergence of the semismooth Newton method follows.

The disadvantage of \eqref{eq:newton_step} is the forward--backward structure in time, which prevents using time-stepping methods for the application of $S_0$ and $S^*$.
For the practical implementation, we therefore eliminate $\delta y$ via the first relation of \eqref{eq:newton_step} as well as $y^k$ via the first relation of \eqref{eq:optimality_red} to obtain the reduced semismooth Newton
step
\begin{equation}
    \label{eq:ssn_red}
    \delta p+S^*S_{0} D_N H_\gamma(p^k)\delta p =-\left(p^k+ S^*(S H_\gamma(p^k)-y^d)\right),
\end{equation}
which allows for time-stepping and matrix-free Krylov methods.
The linear operator on the left-hand side becomes self-adjoint using the
inner product generated by $D_NH_\gamma(p^k)$, and hence we apply a matrix-free CG method for the solution of \eqref{eq:ssn_red} using this inner product; see  \cite[Chapter 3.4.1 and Remark 3.9\,ii)]{PieperDiss}.
We point out that monitoring the convergence of the CG method is still based on the standard Euclidean norm of the residual.

An alternative to \eqref{eq:ssn_red} is a reduction to $\delta y$, which would lead to a reduced Newton matrix that is already self-adjoint with respect to the standard inner product.
We do not follow this approach, because it would lead to a significantly higher dimension of the Newton system (full state space versus control space).
In addition, the iterates would be infeasible, since in this case the state equation is only satisfied for a solution of the reduced optimality system.

To account for the local convergence of Newton methods, we embed the semismooth Newton method within a homotopy strategy for $\gamma$,
where we start with a large $\gamma$ which is successively reduced, taking the previous solution as starting point.
Furthermore, we include a backtracking line search based on the residual norm to improve robustness.

\section{Numerical examples}\label{sec:results}

We now illustrate the behavior of the proposed approach and the structure of the resulting controls using a family of numerical examples.
We consider the heat equation in two dimensions, i.e., $y=S(u)$ solves
\begin{equation}
    y_t - \Delta y = \sum_{i=1}^N\chi_{\omega_i}(x) u_i(t) \quad \text{on }\Omega_T,
\end{equation}
with $T=10$ and $\Omega=(-1,1)^2$, and zero initial and Neumann boundary conditions.
We choose circular control domains $\omega_i = B_{0.1}(x_i)$ with centers
$x_i=(\cos(\varphi_i),\sin(\varphi_i))/\sqrt{2}$ regularly distributed along a circle,
where $\varphi_i=\pi/4+2\pi(i-1)/N$, while the observation domain is $\omobs=B_{0.5}(0)$; see \cref{fig:problem_setting}.
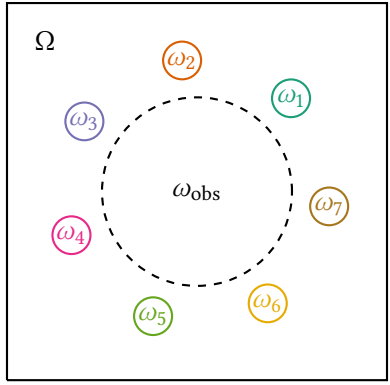
\begin{figure}[t]
    \centering
    \begin{tikzpicture}[scale=2.5]
        \draw [black,thick] (-1,-1)--(-1,1)--(1,1)--(1,-1)--(-1,-1);
        \draw [thick,dashed] (0,0) circle [radius=0.5] node (0,0) {$\omobs$};
        \draw [color=Dark2-A,thick] (45:0.70) circle [radius=0.1] node {$\omega_1$};
        \draw [color=Dark2-B,thick] (45+360/7:0.7) circle [radius=0.1] node {$\omega_2$};
        \draw [color=Dark2-C,thick] (45+2*360/7:0.7) circle [radius=0.1] node {$\omega_3$};
        \draw [color=Dark2-D,thick] (45+3*360/7:0.7) circle [radius=0.1] node {$\omega_4$};
        \draw [color=Dark2-E,thick] (45+4*360/7:0.7) circle [radius=0.1] node {$\omega_5$};
        \draw [color=Dark2-F,thick] (45+5*360/7:0.7) circle [radius=0.1] node {$\omega_6$};
        \draw [color=Dark2-G,thick] (45+6*360/7:0.7) circle [radius=0.1] node {$\omega_7$};
        \draw (-0.8,0.8) node {$\Omega$};
    \end{tikzpicture}
    \caption{Problem setting for $N=7$ control components}\label{fig:problem_setting}
\end{figure}
The desired state $y^d\in L^2(0,T;L^2(\omobs))$ is set to
\begin{align}
    y^d=\sum_{i=1}^N \cos(i+t) \sin^2 \left(2\pi \frac{t}{T}\right) |x-x_i|^2.
\end{align}
The corresponding optimal controls are computed with the semismooth Newton method \eqref{eq:ssn_red},
stopping at a relative tolerance of $10^{-6}$ in the residual
$F_k=\|F(p^k)\|_{L^2(0,T;\R^N)}$, where $F(p)$ denotes the right-hand side of \eqref{eq:ssn_red}.
The backtracking line search starts with a step size $1$ and uses a reduction factor of $1/2$.
The CG method is stopped with a relative tolerance of $10^{-6}$ or after a maximal number of $50$ iterations is reached.
We initialize with $p^0=0$.
The homotopy loop starts from $\gamma=10^{-2}$, reducing $\gamma$ by a factor of $10$ as long as the Newton method converges in a prescribed number of maximal $30$ iterations; the minimal allowed $\gamma$ is $10^{-12}$.
We denote by $\lastgamma$ the resulting smallest $\gamma$ for which the Newton method converged.
Below, we also report on the number $\tau_1$ of control intervals that exhibit perfect switching, i.e., in which at most one control of $N$ is active.
Similarly, we define $\tau_j$ for $j>1$ as the number of time points with $d=j$ in the proximal mapping.

For the spatial discretization, piecewise linear finite elements on an unstructured mesh of $725$ triangular elements (maximal diameter $0.1$) are chosen.
The time discretization is carried out with the cG($1$) Petrov--Galerkin method (which corresponds to the Crank--Nicolson method) and
$201$ equidistant time points. The discrete control is chosen as piecewise constant, i.e., there are  $200$ degrees of freedom per control component.
An implementation of the presented approach in \textsc{Matlab} can be downloaded from \url{https://github.com/clason/multiswitchingcontrol}.

\Cref{fig:other_N:3}, \cref{fig:other_N:5} and \cref{fig:heat2D_N7:1} show the optimal controls for $\alpha=10^{-1}$ and $N=3$, $N=5$ and $N=7$,
respectively (where here and below only active (i.e., nonzero) control components are shown; components that are never active are not listed in the legend, i.e., $u_2$, $u_4$, $u_7$ in these cases, respectively).
All three configurations terminated with $\lastgamma=10^{-12}$, and the optimal controls exhibit perfect switching.
Note that while all control components are discontinuous, the norm $\|u(t)\|_{l^1(\R^N)}$ (which corresponds to the absolute value of the curve plotted in the figure) is continuous.
This is expected since the $L^2$ norm of this quantity appears as a penalty in Problem~\eqref{eq:problem_motiv}.
\pgfplotsset{cycle list/Dark2-7}
\begin{figure}
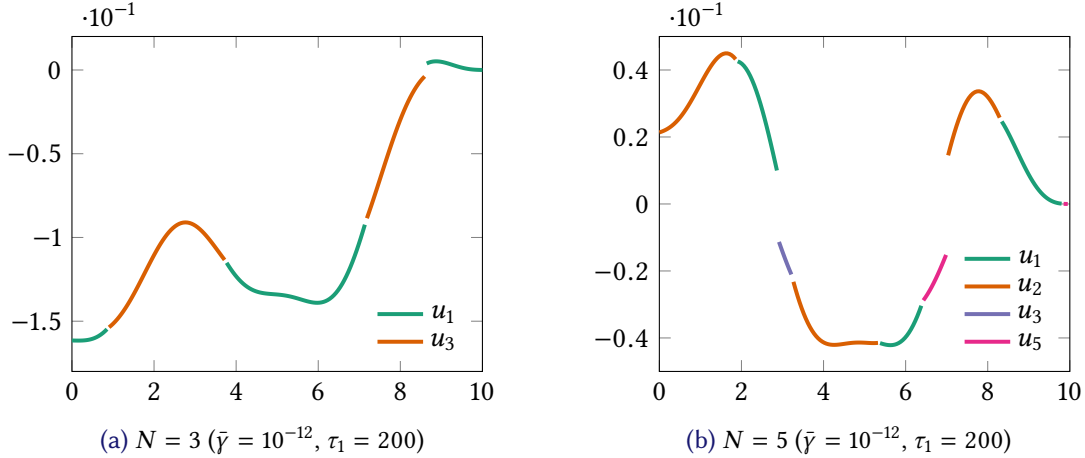

    \centering
    \begin{subfigure}[t]{0.475\textwidth}
        {\input{heat2D_201_h0.1_alpha1e-01_Nf3.tex}}
        \caption{$N=3$ ($\lastgamma=10^{-12}$, $\tau_1=200$)\label{fig:other_N:3}}%
    \end{subfigure}
    \hfill
    \begin{subfigure}[t]{0.475\textwidth}
        {\input{heat2D_201_h0.1_alpha1e-01_Nf5.tex}}
        \caption{$N=5$ ($\lastgamma=10^{-12}$, $\tau_1=200$)\label{fig:other_N:5}}%
    \end{subfigure}
    \caption{Optimal controls for $\alpha=10^{-1}$} \label{fig:other_N}
\end{figure}

Controls for $N=7$ and different values of $\alpha$ are depicted in \cref{fig:heat2D_N7}.
Since $\alpha$ determines the strength of the switching penalty, it is not surprising that the controls are perfectly switching only for sufficiently large $\alpha$, as seen in \cref{fig:heat2D_N7:1} and \cref{fig:heat2D_N7:2}.
For moderate values of $\alpha$ such as $\alpha= 10^{-3}$, there is a single isolated control point where two components are active at the same time ($u_3$ and $u_7$ in \cref{fig:heat2D_N7:3}).
For much smaller $\alpha=10^{-5}$, perfect switching is mostly lost ($\tau_1=17$); however, the control still exhibits a switching property, since in most intervals only two and never more than three components are active ($\tau_2=140$, $\tau_3=43$).
Due to the penalization of the $L^2$ norm in time of the $\ell_1$ norm of the controls, we also observe that the magnitude of the plotted control envelope increases with decreasing $\alpha$ while also becoming less regular with respect to time.
For those controls exhibiting perfect switching, the number of switching points (i.e., of time points $\tau\in (0,T)$ where $\arg\max_j|\bar p_j(s)|\neq\arg\max_j|\bar p_j(t)|$ for $s<\tau<t$ close to $\tau$) is relatively independent of $\alpha$, rising 
only slightly from $10$ points for $\alpha=10^{-1}$ to $11$ and $12$ points for $\alpha = 10^{-2}$ and $\alpha = 10^{-3}$, respectively.
\begin{figure}[t]
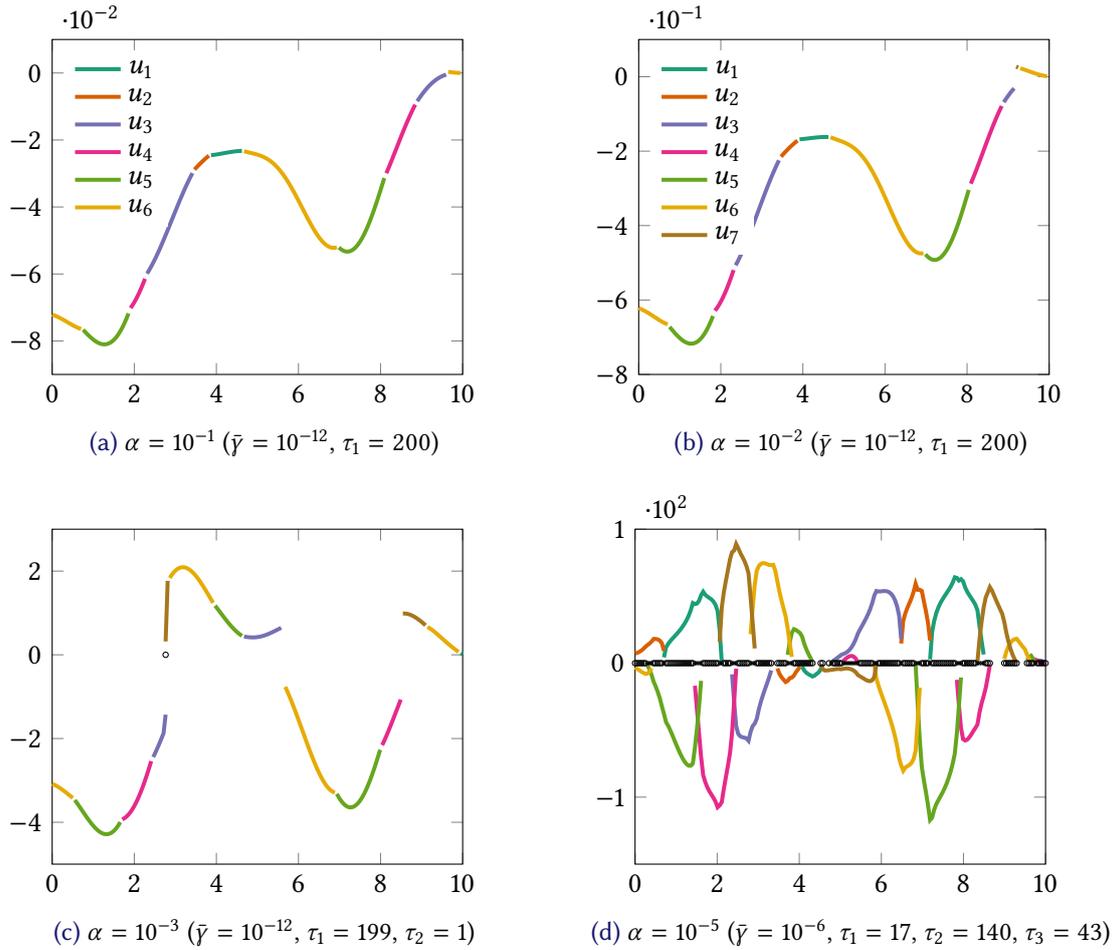

    \centering
    \begin{subfigure}[t]{0.475\textwidth}
        {\input{heat2D_201_h0.1_alpha1e-01_Nf7.tex}}\hfill
        \caption{$\alpha=10^{-1}$ ($\lastgamma=10^{-12}$, $\tau_1=200$)\label{fig:heat2D_N7:1}}%
    \end{subfigure}
    \hfill
    \begin{subfigure}[t]{0.475\textwidth}
        {\input{heat2D_201_h0.1_alpha1e-02_Nf7.tex}}
        \caption{$\alpha=10^{-2}$ ($\lastgamma=10^{-12}$, $\tau_1=200$)\label{fig:heat2D_N7:2}}%
    \end{subfigure}
    
    \bigskip

    \begin{subfigure}[b]{0.475\textwidth}
        {\input{heat2D_201_h0.1_alpha1e-03_Nf7.tex}}\hfill
        \caption{$\alpha=10^{-3}$ ($\lastgamma=10^{-12}$, $\tau_1=199$, $\tau_2=1$)\label{fig:heat2D_N7:3}}%
    \end{subfigure}
    \hfill
    \begin{subfigure}[b]{0.475\textwidth}
        {\input{heat2D_201_h0.1_alpha1e-05_Nf7.tex}}
        \caption{$\alpha=10^{-5}$ ($\lastgamma=10^{-6}$, $\tau_1=17$, $\tau_2=140$, $\tau_3=43$)\label{fig:heat2D_N7:5}}%
    \end{subfigure}
    \caption{Dependence of the controls on $\alpha$ for $N=7$. Control intervals without perfect switching are marked on the $t$-axis} \label{fig:heat2D_N7}
\end{figure}

The dependence of the switching structure on $\alpha$ is illustrated quantitatively in \cref{tab:heat2D_alpha}, where we also report on the effects on the numerical algorithm.
As can be seen, for moderate values of $\alpha$, the homotopy terminates successfully with $\lastgamma=10^{-12}$, and both the number of Newton and CG iterations in the final homotopy and Newton step, respectively, increase only slightly.
For smaller $\alpha$ and $\gamma$, the Newton system becomes increasingly difficult to solve, and the inner CG and Newton iterations terminate due to the maximum number of iterations having been reached, which leads to early termination of the homotopy strategy.
\begin{table}[p]
    \caption{Dependence on $\alpha$ for $N=7$ ($\lastgamma$ is the last successful regularization parameter, \#{}SSN is the number of semismooth Newton steps and \#{}CG is the number of CG steps in the last Newton step for $\lastgamma$)}
    \label{tab:heat2D_alpha}
    \centering
    \begin{tabular}{ccccccc}  
        \toprule
        $\alpha$     & $10^{-1}$  & $10^{-2}$  & $10^{-3}$  & $10^{-4}$  & $10^{-5}$ & $10^{-6}$\\
        \midrule
        $\tau_1$     & $200$      & $200$      & $199$      & $199$      & $17$      & $2$\\
        $\tau_2$     & $0$        & $0$        & $1$        & $1$        & $140$     & $39$\\
        $\tau_3$     & $0$        & $0$        & $0$        & $0$        & $43$      & $146$\\
        \midrule
        $\lastgamma$ & $10^{-12}$ & $10^{-12}$ & $10^{-12}$ & $10^{-12}$ & $10^{-6}$ & $10^{-7}$\\
        \#SSN        & $1$        & $1$        & $4$        & $2$        & $8$       & $15$\\
        \#CG         & $3$        & $3$        & $1$        & $1$        & $28$      & $50$\\
        \bottomrule
    \end{tabular}
\end{table}

To illustrate the convergence rate of the semismooth Newton method, we report in \cref{tab:heat2D_superlinear_convergence} on the convergence history for a single run of the Newton method for  $N=7$, $\alpha=0.01$ and $\gamma=10^{-7}$ without homotopy.
Both the norm of the residual in the reduced first order optimality condition \eqref{eq:ssn_red} and the number of control intervals where the active control components change (denoted by $F_k$ and $s_k$, respectively) show superlinear decay.
It can also be observed that the reduced optimality conditions are satisfied close to machine precision as soon as the active components no longer change.
This finite termination of semismooth Newton methods is typical for problems that are linear apart from the case distinction in $H_\gamma$ and shows the semismooth Newton method's relation to active set methods.
The described behavior is typical for large $\alpha$ or $\gamma$ where no line search is necessary. For moderate values, the iteration first requires some reduced steps before switching to full steps with superlinear convergence and finite termination.
For very small $\alpha$ and $\gamma$, the prescribed maximal number of iterations is insufficient to reach this convergence region, and the method terminates.
\begin{table}
    \caption{Convergence history of the semismooth Newton method for $N=7$, $\alpha=10^{-2}$ and $\gamma=10^{-7}$, where $F_k$ denotes the residual norm and
    $s_k$ denotes the number of control intervals where the active control components change}\label{tab:heat2D_superlinear_convergence}
    \centering
    \begin{tabular}{cccccc} 
        \toprule
        $k$         & $0$                  & $1$                  & $2$                  & $3$                  & $4$\\
        \midrule
        $F_k$       & $3.133\cdot 10^{-2}$ & $3.131\cdot 10^{-2}$ & $1.902\cdot 10^{-3}$ & $8.285\cdot 10^{-6}$ & $3.463\cdot 10^{-12}$\\
        $s_k$       & N/A                  & $200$                & $152$                & $4$                  & $0$ \\
        \bottomrule
    \end{tabular}
\end{table}

The behavior of the homotopy strategy for moderate $\alpha=10^{-3}$ is illustrated in \cref{tab:heat2D_gamma}, where both the number of control intervals $\tau_1,\tau_2,\tau_3$
with one, two or three active control components, respectively, and the number of Newton iterations together with the number of CG iterations in the last Newton step are shown.
For large $\gamma$, the regularization naturally prevents strong switching, which starts to appear for $\gamma<10^{-4}$; for $\gamma<10^{-8}$, the structure of active components no longer changes.
It can also be observed that  the number of Newton and CG steps stays relatively constant throughout the whole iteration.
The situation is different for small $\alpha$; see \cref{tab:heat2D_gamma2}. Since the regularized subdifferential depends on the relation between $\alpha$ and $\gamma$, the switching structure starts to appear later.
Depending on the problem, the limit case for small $\alpha,\gamma$ might have a solution that is not perfectly switching, in which case the unregularized subdifferential is set-valued.
In this example, we observe that for small $\alpha$ and $\gamma$, the Newton steps indeed become increasingly difficult to solve, as shown by the growing number of steps necessary before the iteration terminates.
\begin{table}
    \caption{Dependence on $\gamma$ for $N=7$, $\alpha=10^{-3}$ (\#{}SSN is the number of semismooth Newton steps, \#{}CG is the number of CG steps in the last Newton step)}
    \label{tab:heat2D_gamma}
    \centering
    \begin{tabular}{cccccccccc}  
        \toprule
        $\gamma$ & $10^{-2}$ & $10^{-3}$ & $10^{-4}$ & $10^{-5}$ & $10^{-6}$ & $10^{-8}$ & $10^{-10}$ & $10^{-12}$ \\
        \midrule
        $\tau_1$ & $0$       & $ 0$      & $ 101$    & $189$     & $198$     & $199$     & $199$      & $199$ \\
        $\tau_2$ & $0$       & $45$      & $91$      & $11$      & $2$       & $1$       & $1$        & $1$ \\
        $\tau_3$ & $0$       & $101$     & $8$       & $0$       & $0$       & $0$       & $0$        & $0$ \\
        \midrule
        \#SSN    & $3$       & $4$       & $3$       & $3$       & $3$       & $4$       & $2$        & $4$ \\
        \#CG     & $4$       & $5$       & $5$       & $4$       & $5$       & $4$       & $1$        & $1$ \\
        \bottomrule
    \end{tabular}
\end{table}
\begin{table}
    \caption{Dependence on $\gamma$ for $N=7$, $\alpha=5\cdot10^{-5}$ (\#{}SSN is the number of semismooth Newton steps, \#{}CG is the number of CG steps in the last Newton step)}
    \label{tab:heat2D_gamma2}
    \centering
    \begin{tabular}{cccccccccc}  
        \toprule
        $\gamma$     & $10^{-2}$ & $10^{-3}$ & $10^{-4}$ & $10^{-5}$ & $10^{-6}$ & $10^{-7}$ & $10^{-8}$ \\
        \midrule
        $\tau_1$     & $0$       & $0$      & $0$      & $49$     & $169$     & $186$     & $188$  \\
        $\tau_2$     & $0$       & $0$       & $0$       & $126$     & $28$      & $14$      & $12$ \\
        $\tau_3$     & $0$       & $0$       & $70$      & $19$      & $3$       & $0$       & $0$ \\
        \midrule
        \#SSN        & $3$       & $3$       & $3$       & $3$       & $7$       & $20$      & $21$ \\
        \#CG         & $4$       & $6$       & $8$       & $10$      & $11$      & $15$      & $24$ \\
        \bottomrule
    \end{tabular}
\end{table}

\section{Conclusion}

It is possible to promote optimal controls for parabolic differential equations with switching structure using a convex penalty, avoiding the need to introduce additional decision variables or explicit switching points.
A combination of Moreau--Yosida regularization and a semismooth Newton method allows the efficient numerical computation of controls for an arbitrary number of control components.
Since the proposed switching penalty is independent of the state equation, our approach can be adapted to hyperbolic and nonlinear problems. Furthermore, additional sparsity or control constraints can be incorporated into the presented convex analysis framework.
At the same time, the convexification is unable to guarantee perfect switching for arbitrary choice of $\alpha$.
Analytical issues including the dependency of the optimal solution on $\alpha$, the number of switches for the optimal control, and sufficient conditions ensuring perfect switching of the optimal control depend on the specific structure of the problem setting and are of interest for future research. 

\section*{Acknowledgment}

This work was supported in part by the Austrian Science Fund (FWF) under grant SFB {F}32 (SFB ``Mathematical Optimization and Applications in Biomedical Sciences'').

\printbibliography
\end{document}